\documentclass{article}

\usepackage[scale=0.7]{geometry}
\usepackage{amsfonts}
\usepackage{mathrsfs}
\usepackage{amsmath}
\usepackage{amssymb}
\usepackage{enumerate}
\usepackage{graphicx}%
\numberwithin{equation}{section}
\newtheorem{theorem}{Theorem}[section]

\newtheorem{lemma}[theorem]{Lemma}

\newtheorem{proposition}[theorem]{Proposition}
\newtheorem{remark}[theorem]{Remark}

\newenvironment{proof}[1][Proof]{\noindent\textbf{#1.} }{\ \rule{0.5em}{0.5em}}

\usepackage{cite}
\begin{document}

\begin{center}
\Large \textbf{A generalized version of Holmstedt’s formula for the $K$-functional   \\}
\vskip0.5cm
\large{\textbf{Irshaad Ahmed}${}^{1}$,   \textbf{Alberto  Fiorenza${}^{2}$}} and  \textbf{Amiran Gogatishvili${}^{3}$}\\
\vskip0.5cm
\small{${}^{1}$Department of Mathematics, Sukkur IBA University, Sukkur, Pakistan.\\irshaad.ahmed@iba-suk.edu.pk

\small${}^{2}$Universit\`a di Napoli Federico II, Dipartimento di
Architettura, via Monteoliveto, 3, 80134 - Napoli,   Italy and
Consiglio Nazionale delle Ricerche, Istituto per le Applicazioni del
Calcolo ``Mauro Picone", Sezione di Napoli, via Pietro Castellino,
111, 80131 - Napoli, Italy.\\fiorenza@unina.it

\small${}^{3}$Institute of  Mathematics of the   Czech Academy of Sciences - \v Zitn\'a, 115 67 Prague 1,  Czech Republic.\\gogatish@math.cas.cz }
\end{center}

\begin{abstract}  Let $(A_0, A_1)$ be a compatible couple of quasi-normed spaces, and let $\Phi_0$ and $\Phi_1$ be two general parameters of $K$-interpolation method. We compute  $K$-functional for the couple $((A_0,A_1)_{\Phi_0}, (A_0, A_1)_{\Phi_1})$ in terms of $K$-functional for the couple $(A_0, A_1)$.
\end{abstract}



\section{Introduction}

  Let $(A_0,A_1)$  be a compatible couple of quasi-normed spaces.   For each $ f\in A_0+A_1$ and $t>0$, the Peetre's  $K$-functional is defined by
\begin{eqnarray*}
K(t,f)&=&K(t,f;A_0,A_1)\\
&=&\inf\{\|f_0\|_{A_0}+t\|f_1\|_{A_1}:\;f_0 \in A_0, \; f_1 \in A_1,\; f=f_0+f_1\}.
\end{eqnarray*}
  Let $\Phi$ be a quasi-normed space of Lebesgue measurable functions defined on $(0,\infty)$ with monotone quasi-norm, that is, $|g|\leq |h|$ implies  $\|g\|_{\Phi}\leq \|h\|_{\Phi}$.  Assume that $t\mapsto \min(1,t) \in \Phi.$ The general $K$-interpolation ${\bar{A}}_{\Phi}=(A_0, A_1)_{\Phi}$ is formed of those $f\in A_0+A_1$ for which the quasi-norm
   $$
\|f\|_{\bar{A}_{\Phi}}=\|K(t,f)\|_{\Phi}
$$
is finite; see \cite{BK}. We call $\Phi$  as  general parameter of $K$-interpolation method. \\

Let $(\theta, q)\in([0,1]\times[1,\infty])\setminus(\{0,1\}\times[1,\infty]).$ In particular, let $\Phi$ be given by
$$\|g\|_{\Phi}=\left(\int_0^\infty[t^{-\theta}|g(t)|]^q\frac{dt}{t}\right)^{1/q},$$
with the usual modification when $q=\infty.$ Then we put ${\bar{A}}_{\Phi}={\bar{A}}_{\theta, q}$, and ${\bar{A}}_{\theta, q}$  is the classical scale of $K$-interpolation spaces (see \cite{BL, T, BS}). In 1970, Tord Holmstedt computed, up to equivalence of constants,   $K$-functional for the couple $({\bar{A}}_{\theta_0,q_0}, {\bar{A}}_{\theta_1,q_1})$ in terms of $K$-functional for the couple $(A_0, A_1)$. The precise formula is as follows.   Let $0<q_0,q_1\leq \infty$ and $0<\theta_0< \theta_1<1$. Then  for all  $f\in A_0+A_1$ and for all $t>0$,  we have
\begin{eqnarray*}
K(t^{\theta_1-\theta_0},f;{\bar{A}}_{\theta_0,q_0}, {\bar{A}}_{\theta_1,q_1})&\approx&\|u^{-\theta_0-1/{q_0}}K(u,f)\|_{q_0, (0,t)}\\
&&+t^{\theta_1-\theta_0}\|u^{-\theta_1-1/{q_1}}K(u,f)\|_{q_1, (t,\infty)},
\end{eqnarray*}
(see \cite[Theorem 2.1]{Ho}).
 Since then several authors have obtained generalizations or variants of Holmstedt's formula (see, for instance, \cite{AEEK, AE, AFG, AKA, D1, D2, D3, DFS1, DFS2, EOP, FS1, FS3, FS4, FS5, FS6, GOT, He, Pe }).\\

  Let $\Phi_0$ and $\Phi_1$ be two general parameters of $K$-interpolation method.   In this paper,  we compute the  $K$-functional for the couple $(\bar{A}_{\Phi_0}, \bar{A}_{\Phi_1})$ in terms of $K$-functional for the couple $(A_0, A_1)$ under certain appropriate  separation conditions imposed on $\Phi_0$ and $\Phi_1$. Our computation provides a general frame work to derive Holmstedt-type estimates. In particular,  all  those  Holmstedt type formulae (except the one  in \cite[Theorem 3.5]{AFG}) contained in the papers  \cite{ AE, AFG, AKA, D1, D2, D3, DFS2, EOP, FS1, FS3, FS4, FS5, FS6, GOT, He, Pe}, which express the $K$-functional of the couple $(\bar{A}_{\Phi_0}, \bar{A}_{\Phi_1})$,  for special cases of $\Phi_0$ and $\Phi_1$,  in terms of that of $(A_0, A_1)$,   can be derived immediately  by our general computation.

\section{A notation and slowly varying functions}
\subsection{Notation}

  Let  $f$ and $g$ be two positive functions defined on $(0,\infty)$.  We write $$f(t) \lesssim g(t),\;\;t>0, $$
if there exists a constant  $c>0$ such that $$f(t) \leq c g(t),\;\; t>0.$$
We simply write $$f(t) \approx g(t), \;\;t>0, $$ if
$$f(t) \lesssim  g(t),\;\;  t>0,\; \text{and}\; g(t) \lesssim f(t), \;\;t>0. $$
\subsection{Slowly varying functions}
Let $b:(0,\infty)\rightarrow (0,\infty)$  be a Lebesgue measurable function. Following \cite{GOT}, we say $b$ is slowly varying on $(0,\infty)$ if for every  $\varepsilon>0,$  there  are  positive  functions $g_{\varepsilon}$ and $g_{-\varepsilon}$ on $(0,\infty)$ such that $g_{\varepsilon}$ is non-decreasing and  $g_{-\varepsilon}$ is non-increasing, and we have
$$
t^{\varepsilon}b(t)\approx g_{\varepsilon}(t)\; \; \text{and}\;\;  t^{-\varepsilon}b(t)\approx g_{-\varepsilon}(t)\;\; \text{for all }\; t\in (0,\infty).
$$

We denote the class of all slowly varying functions by $SV.$ Let $\mathbb{A}=(\alpha_0,\alpha_\infty)\in \mathbb{R}^2.$ Define
$$\ell^\mathbb{A}(t)=\left\{
\begin{array}
[c]{cc}%
(1-\ln t)^{\alpha_0}, & 0<t\leq 1,\\
& \\
(1+\ln t)^{\alpha_\infty}, &
t>1,
\end{array}
\right.
$$
 Then $\ell^\mathbb{A}\in SV.$ In addition, the class $SV$ contains compositions of appropriate log-functions, $\exp |\log t|^\alpha$ with $\alpha\in (0,1)$, etc.

We collect in next Proposition some elementary properties of slowly varying functions, which we will need in Section 4.  The proofs of these assertions can be carried out as in \cite[Lemma $2.1$]{GOT} or \cite[Proposition $3.4.33$]{EE}.
\begin{proposition}\label{svp1} Given  $b$, $b_1$, $ b_2\in SV$, the following assertions hold:

\begin{enumerate}[(i)]
    \item $b_1 b_2 \in SV$ and  $b^{r} \in SV$ for each $ r\in \mathbb{R}.$

    \item  If $\alpha>0$, then

$$\int_0^t s^{\alpha}b(s)\frac{ds}{s}\approx t^{\alpha}b(t),\quad t>0.$$

\item If $\alpha>0$, then
$$\int_t^\infty s^{-\alpha}b(s)\frac{ds}{s}\approx t^{-\alpha}b(t),\quad t>0.$$

\item Assume that
$$\int_0^1 b(t)\frac{dt}{t}<\infty,$$
and set
$${B}(t)=\int_0^t b(s)\frac{ds}{s},\quad t>0.$$
Then ${B}\in SV,$ and $b(t)\lesssim {B}(t),\; t>0.$

\item Assume that
$$\int_1^\infty b(t)\frac{dt}{t}<\infty,$$
and set
$$\tilde{B}(t)=\int_t^\infty b(s)\frac{ds}{s},\quad t>0.$$
Then $\tilde{B}\in SV,$ and $b(t)\lesssim \tilde{B}(t),\; t>0.$
\end{enumerate}

\end{proposition}

\section{Main result}

  Let $\Phi_0$ and $\Phi_1$ be two general parameters of $K$-interpolation method. In order to formulate our results we need the following conditions on a positive weight $\rho$ defined on $(0,\infty).$

$$
(C_1)\;\;\;\;\;\;\;\;\;\;\;\;\;\;\;\;\;\;\;\;\;\; \frac{t \|\chi_{(t,\infty)}\|_{\Phi_0}}{\|\min(u,t)\|_{\Phi_1}}\lesssim \rho(t) \lesssim    \frac{ \|\min(u,t)\|_{\Phi_0}}{\|u\chi_{(0,t)}(u)\|_{\Phi_1}},\; \;\; t>0.
$$

$$
(C_2)\;\;\;\;\;\;\;\;\;\;\;\;\;\;\;\;\;\;\;\;\;\;\;\;\;\;\;\;\left\|\chi_{(0,t)}(u)u\|\min(s,u)\|^{-1}_{\Phi_1}\right\|_{\Phi_0}\lesssim \rho(t),\;\; t>0.
$$

$$
(C_3)\;\;\;\;\;\;\;\;\;\;\;\;\;\;\;\;\;\;\;\;\;\;\;\;\;\;\left\|\chi_{(t,\infty)}(u)u\|\min(s,u)\|^{-1}_{\Phi_0}\right\|_{\Phi_1}\lesssim \frac{1}{\rho(t)},\;\; t>0.
$$

$$
(C_4)\;\;\;\;\;\;\;\;\;\;\;t\|\chi_{(t,\infty)}\|_{\Phi_0}\lesssim \|u\chi_{(0,t)}(u)\|_{\Phi_0} +t\rho(t)\|\chi_{(t,\infty)}\|_{\Phi_1},\;\; t>0.
$$

\begin{lemma} \label{L1}\cite[Theorem 4 (Case 1)]{AKA}
Let $(A_0, A_1)$ be a compatible couple of quasi-normed spaces, and let $\Phi_0$ and $\Phi_1$ be two general parameters of $K$-interpolation method. Assume that a weight $\rho$ satisfies the conditions $(C_1)$, $(C_2)$ and $(C_3)$. Then, for  all  $f\in A_0+A_1$ and for all $t>0$, the following two-sided estimate
\begin{align}\label{mtpe1}
K(\rho(t),f;{\bar{A}}_{\Phi_0}, {\bar{A}}_{\Phi_1}) &\approx  \|\chi_{(0,t)}(u)K(u,f)\|_{\Phi_0}+ \rho(t) \|\chi_{(t,\infty)}(u)K(u,f)\|_{\Phi_1}\notag\\
 &\;\;\;\;+\|\chi_{(t,\infty)}\|_{\Phi_0}K(t,f)+\rho(t)\|u\chi_{(0,t)}(u)\|_{\Phi_1}\frac{K(t,f)}{t},
\end{align}
holds.

\end{lemma}

The main result of this paper is following.

\begin{theorem}\label{T1GHF}
Let $(A_0, A_1)$ be a compatible couple of quasi-normed spaces, and let $\Phi_0$ and $\Phi_1$ be two general parameters of $K$-interpolation method.   Put
  $$
\rho(t)=\frac{\|\min(u,t)\|_{\Phi_0}}{\|\min(u,t)\|_{\Phi_1}},\;\; t>0.
$$
\begin{enumerate}[(i)]
    \item Assume that $\rho$ satisfies the conditions  $(C_2)$ and $(C_3)$.
Then, for  all  $f\in A_0+A_1$ and for all $t>0$, the following two-sided estimate

\begin{align}\label{T1GHFe1}
K(\rho(t),f;{\bar{A}}_{\Phi_0}, {\bar{A}}_{\Phi_1}) &\approx  \|\chi_{(0,t)}(u)K(u,f)\|_{\Phi_0}
+\rho(t)\|\chi_{(t,\infty)}(u)K(u,f)\|_{\Phi_1}\notag\\
 &\;\;\;\;+\|\chi_{(t,\infty)}\|_{\Phi_0}K(t,f),
\end{align}
holds.
\item Assume that $\rho$ satisfies the conditions  $(C_2)$, $(C_3)$ and $(C_4)$. Then, for  all  $f\in A_0+A_1$ and for all $t>0$, the following two-sided estimate
\begin{equation}\label{T1GHFe2}
 K(\rho(t),f;{\bar{A}}_{\Phi_0}, {\bar{A}}_{\Phi_1}) \approx  \|\chi_{(0,t)}(u)K(u,f)\|_{\Phi_0}
+\rho(t)\|\chi_{(t,\infty)}(u)K(u,f)\|_{\Phi_1},
\end{equation}
holds.
\end{enumerate}

\end{theorem}
 \begin{proof}
 First we derive the formula (\ref{T1GHFe1}) using Lemma \ref{L1}. Since (for $j=0,1$)
 $$\|\min(u,t)\|_{\Phi_j}\approx \|u\chi_{(0,t)}(u)\|_{\Phi_j} + t\|\chi_{(t,\infty)}\|_{\Phi_j},\;\; t>0,$$
  the condition  $(C_1)$  is satisfied trivially by our choice of $\rho.$ Next note that
\begin{equation}\label{mtpe21}
\rho(t)\|u\chi_{(0,t)}(u)\|_{\Phi_1}\frac{K(t,f)}{t}\lesssim \|\chi_{(t,\infty)}\|_{\Phi_0} K(t,f)+ \|u\chi_{(0,t)}(u)\|_{\Phi_0}\frac{K(t,f)}{t}.
\end{equation}
Since $t\mapsto K(t,f)/t$ is non-increasing (see \cite[Proposition 1.2]{BS}), we have
\begin{equation*}
 \|\chi_{(0,t)}(u)K(u,f)\|_{\Phi_0}\geq \frac{K(t,f)}{t}\|u\chi_{(0,t)}(u)\|_{\Phi_0},
\end{equation*}
therefore, (\ref{mtpe21}) reduces to
\begin{equation}\label{mtpe31}
\rho(t)\|u\chi_{(0,t)}(u)\|_{\Phi_1}\frac{K(t,f)}{t}\lesssim  \|\chi_{(t,\infty)}\|_{\Phi_0}K(t,f) +  \|\chi_{(0,t)}(u)K(u,f)\|_{\Phi_0}.
\end{equation}
Finally, in view of (\ref{mtpe31}), the estimate   (\ref{T1GHFe1})  follows from (\ref{mtpe1}).\\

Next we derive (\ref{T1GHFe2}).  Since $t\mapsto K(t,f)$  is non-decreasing (see \cite[Proposition 1.2]{BS}) and  $t\mapsto K(t,f)/t$ is non-increasing, we have
$$\|\chi_{(0,t)}(u)K(u,f)\|_{\Phi_0}
+\rho(t)\|\chi_{(t,\infty)}(u)K(u,f)\|_{\Phi_1}\geq \frac{K(t,f)}{t}\|u\chi_{(0,t)}(u)\|_{\Phi_0} +\rho(t) K(t,f)\|\chi_{(t,\infty)}\|_{\Phi_1},$$
combining above inequality with the condition $(C_4)$, we arrive at
$$
\|\chi_{(t,\infty)}\|_{\Phi_0}K(t,f)\lesssim \|\chi_{(0,t)}(u)K(u,f)\|_{\Phi_0} +\rho(t)\|\chi_{(t,\infty)}(u)K(u,f)\|_{\Phi_1},\;\; t>0,
$$
now inserting the above estimate in (\ref{T1GHFe1}) yields (\ref{T1GHFe2}). The proof is complete.
\end{proof}
 \begin{remark}\label{R2}{\em
The  essential difference between  the general computation in our main result and that of \cite[Theorem 4]{AKA} is that the   weight $\rho$ is explicitly given  in our main result.
}
\end{remark}

\section{Scope of the main result}
 
 As mentioned in the Introduction,  almost all  those  Holmstedt-type estimates, contained in the papers  \cite{ AE, AFG, AKA, D1, D2, D3, DFS2, EOP, FS1, FS3, FS4, FS5, FS6, GOT, He, Pe}, which express the couple $(\bar{A}_{\Phi_0}, \bar{A}_{\Phi_1})$,  for special cases of $\Phi_0$ and $\Phi_1$,  in terms of $(A_0, A_1)$,   can be derived immediately  by our general computation. In this section, we illustrate this in some instances. \\

 \begin{enumerate}[(i)]

 \item Let $0< q_0, q_1\leq \infty$, $0\leq \theta_0, \theta_1\leq 1$, and let $b_j$ ($j=0,1$) be a slowly varying function. Define
 $$\|g\|_{\Phi_{ \theta_j,q_j; b_j}}=\left(\int_0^\infty \left[t^{-\theta_j}b_j(t)|g(t)|\right]^{q_j}\frac{dt}{t}\right)^{1/{q_j}},$$
with the usual modification when $q_j=\infty\,\, (j=0,1).$  First consider the case $0<\theta_0< \theta_1<1$. Thanks to Proposition \ref{svp1}, it is easy to compute that
$$\|\min(u,t)\|_{\Phi_{\theta_j,q_j;b_j}}\approx t^{1-\theta_j}b_j(t),\;\; t>0.
$$ The conditions $(C_2)$, $(C_3)$ and $(C_4)$ are also met again thanks to Proposition \ref{svp1}. As a consequence of Theorem \ref{T1GHF} (ii), we recover an extension of  Holmstedt's formula contained in \cite[Theorem 3.1 (a)]{GOT}. When $b_j\equiv 1$, we get back the classical Holmstedt's formula.\\

Next we consider the case when $\theta_0=0$ and $\theta_1=1$. Again thanks to Proposition \ref{svp1}, the conditions $(C_2)$ and $(C_3)$ are met, and also we have the following computations
 $$\|\min(u,t)\|_{\Phi_{0, q_0;b_0}}\approx t\left(\int_t^\infty b_0^{q_0}(s)\frac{ds}{s}\right)^{1/q_0},\;\; t>0,$$
 and
  $$\|\min(u,t)\|_{\Phi_{1, q_1;b_1}}\approx \left(\int_0^t b_1^{q_1}(s)\frac{ds}{s}\right)^{1/q_1},\;\; t>0.$$
  Let us  put $\bar{A}_{j,q_j:b_j}=\bar{A}_{\Phi_{j, q_j;b_j}}$ ($j=0,1$). As a consequence of Theorem \ref{T1GHF} (i), for all $f\in A_0+A_1$ and $t>0$, we get the following formula

\begin{eqnarray*}
K\left(\rho(t),f;\bar{A}_{0,q_0:b_0}, \bar{A}_{1,q_1:b_1}\right)&\approx&    \left(\int_0^t b_0^{q_0}(s)K^{q_0}(s,f)\frac{ds}{s}\right)^{1/q_0} + \rho(t)\left(\int_t^\infty b_1^{q_1}(s)K^{q_1}(s,f)\frac{ds}{s}\right)^{1/q_1}\\
&&+\left(\int_t^\infty b_0^{q_0}(s)\frac{ds}{s}\right)^{1/q_0}K(t,f),
\end{eqnarray*}
where
$$\rho(t)=\frac{t\left(\int_t^\infty b_0^{q_0}(s)\frac{ds}{s}\right)^{1/q_0}}{\left(\int_0^t b_1^{q_1}(s)\frac{ds}{s}\right)^{1/q_1}}, \;\; t>0.$$
 This formula also follows from  \cite[Example 5]{AKA}, however  computations in \cite[Example 5]{AKA} are tedious as compared to our straightforward computations.\\

 Next we consider the case  when $\theta_0=\theta_1=0$ with $q_0\neq q_1.$ This time (thanks to Proposition \ref{svp1}) we have (for $j=0,1$)
  $$\|\min(u,t)\|_{\Phi_{j, q_j;b_j}}\approx t\left(\int_t^\infty b_j^{q_j}(s)\frac{ds}{s}\right)^{1/q_j},\;\; t>0. $$
The condition $(C_4)$ is met while the conditions $(C_2)$ and $(C_3)$ turn into
$$
\left(\int_{0}^{t}\frac{b_0^{q_0}(x)}{\left(\int_{x}^{\infty}b_1^{q_1}(s)\frac{ds}{s}\right)^{q_0/q_1}}\frac{dx}{x}\right)^{1/q_0}\lesssim \frac{\left(\int_{t}^{\infty}b_0^{q_0}(s)\frac{ds}{s}\right)^{1/q_0}}{\left(\int_{t}^{\infty}b_1^{q_1}(s)\frac{ds}{s}\right)^{1/q_1}},\;\; t>0,
$$
and
$$
\left(\int_{t}^{\infty}\frac{b_1^{q_1}(x)}{\left(\int_{x}^{\infty}b_0^{q_0}(s)\frac{ds}{s}\right)^{q_1/q_0}}\frac{dx}{x}\right)^{1/q_1}\lesssim \frac{\left(\int_{t}^{\infty}b_1^{q_1}(s)\frac{ds}{s}\right)^{1/q_1}}{\left(\int_{t}^{\infty}b_0^{q_0}(s)\frac{ds}{s}\right)^{1/q_0}},\;\; t>0.
$$
As a consequence of Theorem \ref{T1GHF} (ii), we recover the Holmstedt-type estimate contained in \cite[Theorem 3.2]{AFG}. It is worthy of mention that it has been established in \cite[Theorem 3.2]{AFG} that a sufficient condition for the previous two estimates to be held is that the following function
$$t\mapsto \frac{\left(\int_{t}^{\infty}b_0^{q_0}(s)\frac{ds}{s}\right)^{1/q_0(1+\epsilon)}}{\left(\int_{t}^{\infty}b_1^{q_1}(s)\frac{ds}{s}\right)^{1/q_1}} $$
is equivalent to a nondecreasing function for some $\epsilon>0.$

 \item Next we turn to the    papers \cite{FS4, FS5, FS6}.  Let  $0 \leq \theta \leq 1$, and  let $a$, $b$ be slowly varying functions.    Let $E$, $F$ be  rearrangement invariant spaces (see \cite{BS}), and let $\tilde{E}$ and  $\tilde{F}$ be the corresponding rearrangement invariant spaces with respect to homogeneous measure $dt/t$ (see \cite[p. 135]{FS1}).  Define
    $$\|g\|_{\Phi_{\theta,b,E}}=\left\|t^{-{\theta}}b(t)g(t)\right\|_{\tilde{E}},$$

     $$\|g\|_{\Phi_{\mathcal{L};\theta,b, E,a,F}}=\left\|b(t)\left\|s^{-{\theta}}a(s)g(s)\right\|_{\tilde{F}(0,t)}\right\|_{\tilde{E}}.$$
     and
     $$\|g\|_{\Phi_{\mathcal{R};\theta,b, E,a,F}}=\left\|b(t)\left\|s^{-{\theta}}a(s)g(s)\right\|_{\tilde{F}(t,\infty)}\right\|_{\tilde{E}},$$

     Using the estimates in \cite[Lemma 2.3]{FS1}, we can easily make the following straightforward computations. For $0< \theta < 1,$ we have
     $$\|\min(u,t)\|_{\Phi_{\theta,q,E}}\approx t^{1-\theta}b(t),\;\; t>0,
    $$
    $$\|\min(u,t)\|_{\Phi_{\mathcal{L};\theta,b, E,a,F}}\approx t^{1-\theta}a(t)\|b\|_{\tilde{E}(t,\infty)},\;\; t>0,
    $$
    and
    $$\|\min(u,t)\|_{\Phi_{\mathcal{R};\theta,b, E,a,F}}\approx t^{1-\theta}a(t)\|b\|_{\tilde{E}(0,t)},\;\; t>0.
    $$
    And also we have

    $$\|\min(u,t)\|_{\Phi_{0,q,E}}\approx t \|b\|_{\tilde{E}(t,\infty)},\;\; t>0,
    $$
    and
    $$\|\min(u,t)\|_{\Phi_{1,q,E}}\approx \|b\|_{\tilde{E}(0,t)},\;\; t>0.
    $$

    Now  using the above computations and estimates in  \cite[Lemma 2.3]{FS1}, we can  immediately derive, from Theorem \ref{T1GHF},  all the Holmstedt-type estimates contained in \cite[Theorem 3.1]{FS5}, \cite[Theorem 3.2]{FS5}, \cite[Theorem 3.4]{FS5}, \cite[Theorem 3.1]{FS4}, \cite[Theorem 3.2]{FS4},  \cite[Theorem 3.1]{FS6}, \cite[Theorem 3.2]{FS6}, \cite[Theorem 3.3]{FS6} and  \cite[Theorem 3.4]{FS6}.  It is worthy of mention that authors of above mentioned papers have to carry out   a lot of computations in  each separate case. While, on the other hand, our main result provides a general framework where all these separate cases can be treated simultaneously by means of handful computations.

 \end{enumerate}

{\bf Acknowledgements.} The research of A. Gogatishvili was partially supported by the Czech Academy of Sciences (RVO 67985840), by Czech Science Foundation (GAČR), grant no: 23-04720S, and by Shota Rustaveli National Science Foundation of Georgia (SRNSFG), grant no: FR21-12353.

\end{document}